\documentclass[12pt]{article}
\usepackage{amsmath,amssymb,amsthm, amsfonts}
\usepackage{hyperref}
\usepackage{graphicx}
\usepackage{array, tabulary}
\usepackage{url}
\usepackage{dsfont} % needed for double-struck 1
\usepackage{color}
\usepackage{subcaption}
\usepackage{enumitem}
\definecolor{red}{rgb}{1,0,0}

\definecolor{blue}{rgb}{0,0,1}

\definecolor{green}{rgb}{0,.6,0}

\usepackage{float}

\usepackage{tikz}

\newtheorem{thm}{Theorem}[section]
\newtheorem{cor}[thm]{Corollary}
\newtheorem{lem}[thm]{Lemma}
\newtheorem{prop}[thm]{Proposition}
\newtheorem{conj}[thm]{Conjecture}
\newtheorem{obs}[thm]{Observation}

\theoremstyle{definition}
\newtheorem{rem}[thm]{Remark}

\theoremstyle{definition}
\newtheorem{defn}[thm]{Definition}

\theoremstyle{definition}

\newcommand{\bit}{\begin{itemize}}
\newcommand{\eit}{\end{itemize}}
\newcommand{\ben}{\begin{enumerate}}
\newcommand{\een}{\end{enumerate}}
\newcommand{\beq}{\begin{equation}}
\newcommand{\eeq}{\end{equation}}
\newcommand{\bea}{\begin{eqnarray}} 
\newcommand{\eea}{\end{eqnarray}}
\newcommand{\bpf}{\begin{proof}}
\newcommand{\epf}{\end{proof}\ms}
\newcommand{\bmt}{\begin{bmatrix}}
\newcommand{\emt}{\end{bmatrix}}
\newcommand{\ms}{\medskip}

\newcommand{\noi}{\noindent}

\newcommand{\beqs}{\begin{equation*}} % * means no number
\newcommand{\eeqs}{\end{equation*}}
\newcommand{\beas}{\begin{eqnarray*}}
\newcommand{\eeas}{\end{eqnarray*}}

\title{On the zero forcing number of the complement of graphs with forbidden subgraphs}
\author{Emelie Curl
\thanks{Department of Mathematics, Statistics, and Computer Science, Hollins University, Roanoke, VA, USA (curlej@hollins.edu)}
\and Shaun Fallat
\thanks{Department of Mathematics and Statistics, University of Regina, Regina, Saskatchewan, CA (shaun.fallat@uregina.ca)}
\and Ryan Moruzzi Jr
\thanks{Department of Mathematics, California State University East Bay, Hayward, CA, USA (ryan.moruzzi@csueastbay.edu)}~\thanks{Corresponding Author}
\and Carolyn Reinhart
\thanks{Department of Mathematics and Statistics, Swarthmore College, Swarthmore, PA, USA (creinha1@swarthmore.edu)}
\and Derek Young
\thanks{Department of Mathematics and Statistics, Mount Holyoke College, South Hadley, MA, USA (dyoung@mtholyoke.edu)}}

\date{\today}

\begin{document}

\maketitle

\begin{abstract}
Motivated in part by an observation that the zero forcing number for the complement of a tree on $n$ vertices is either $n-3$ or $n-1$ in one exceptional case, we consider the zero forcing number for the complement of more general graphs under some conditions, particularly those that do not contain complete bipartite subgraphs. We also move well beyond trees and completely study all of the possible zero forcing numbers for the complements of unicyclic graphs and cactus graphs.
\end{abstract}

\noi {\bf Keywords} 
Zero forcing number, complements, bipartite graphs, trees, unicyclic graphs, cactus graphs

\noi{\bf AMS subject classification} 05C50, 05C76.

\section{Introduction}
Zero forcing is an iterative graph coloring process where an initial set of blue vertices eventually color all vertices of a graph using a color change rule. The zero forcing number of a graph was first defined in 2008 by the AIM Minimum Rank Special Graphs Work Group, as a method for bounding the maximum nullity of a graph \cite{AIM}. In \cite{AIM}, the authors determine the zero forcing number of the complement of a tree by proving that the zero forcing number is equal to the maximum nullity for the complement of a tree. However, this proof does not directly involve properties of zero forcing. Since 2008, zero forcing has been a topic of much study, and has grown in to a graph parameter of independent interest (for example, see \cite{B-zf1, B-zf2, b-q, fh, h,l,r}). In this paper, we determine the zero forcing number of the complements of graphs through direct zero forcing arguments, including a new proof of the zero forcing number of the complement of a tree which does not involve maximum nullity. This result follows from a more general result, which gives a lower bound for the zero forcing number of graphs in terms of the smallest complete bipartite graph which they do not contain as a subgraph.

Let $G$ be a graph on $n$ vertices with vertex set $V(G)$ and edge set $E(G)$. The set of vertices adjacent to a vertex $v\in V(G)$ is called the neighborhood of $v$, denoted $N_G(v)$. The \emph{complement} of a graph $\overline{G}$  is the graph with vertex set $V(G)$ and edge set $E(\overline{G})=\{v_iv_j \, : v_i,v_j\in V(G) \, \text{and} \, v_iv_j\not\in E(G)\}$. An \emph{independent set} is a set of vertices such that no two vertices in the set are adjacent. The \emph{complete bipartite graph} $K_{r,s}$ is the graph on $r+s$ vertices such that $V(G)=A\cup B$, where $A$ and $B$ are independent sets of size $r$ and $s$, respectively, and $E(G)=\{v_iv_j : v_i\in A \, \text{and} \, v_j\in B\}$. Suppose $H$ is an induced subgraph of a graph $G$. Then it follows easily from the definitions of graph complements and induced subgraphs that $\overline{H}$ is an induced subgraph of $\overline{G}$.

In the zero forcing process, denote an initial set of blue vertices $B\subset V(G)$ and let the remaining vertices of the graph be white. The \emph{color change rule} states that a blue vertex $b\in B$ can \emph{force} a white vertex $w$ to become blue if $w$ is the only white neighbor of $b$. A \emph{zero forcing set} is any initial set of blue vertices in a graph such that repeated application of the color change rule results in every vertex of the graph being blue. The \emph{zero forcing number} of a graph $G$, denoted $\operatorname{Z}(G)$ is the minimum size of a zero forcing set. A \emph{minimum zero forcing set} of a graph $G$ is a zero forcing set $B$ such that $|B|=\operatorname{Z}(G)$. 

In each application of the color change rule, we assume all possible forces occur simultaneously. Note that for each white vertex which is forced, there may be multiple blue vertices which could have forced it. However, we can assume in this case that one of the blue vertices was chosen to force the white vertex. A \emph{forcing chain} is a sequence of $k$ vertices in the graph $(v_1,v_2,\dots, v_k)$ such that during the zero forcing process, $v_{i}$ forces $v_{i+1}$ for all $1\leq i\leq k-1$ and $1\leq k\leq n$. Each zero forcing set and each choice of forces during the applications of the color change rule results in a set of forcing chains. Note that the number of chains is equal to the size of the zero forcing set and the chains are all disconnected. The \emph{reversal} of a forcing chain $(v_1,v_2,\dots, v_k)$ is the sequence of vertices $(v_k,v_{k-1},\dots, v_1)$. It was shown in \cite{HHKMWY12} that the set of forcing chain reversals are themselves a set of zero forcing chains for the same graph, with the last vertices in each zero forcing chain forming a new zero forcing set of the same size.

The zero forcing number of the complement of a graph has previously been studied in \cite{EKY15}. In this paper, the authors show that the zero forcing number of a graph on $n$ vertices is at most $n-3$ as long as the complement of the graph is connected. They also bound the sum of the zero forcing numbers of a graph and it complement in terms of $n$ and the minimum degree.

In this paper, we begin in Section \ref{sec:KeyResult} by providing a lower bound on the zero forcing number of the complement of a graph which does not contain $K_{r,s}$ as a subgraph. 
%In Section \ref{bipartitesec}, we apply this result to bipartite graphs, bounding the zero forcing number of the complement of a complete bipartite graph from above and below, and showing that each number in this interval can be achieved. We also provide a new proof of the characterization of the zero forcing number of the complement of a tree. 
In Section \ref{sec:uni}, the zero forcing number of the complement of unicyclic graphs on $n$ vertices is determined to be $n-2$, $n-3$ or $n-4$ and we provide a complete characterization of when each of these values is achieved. In Section \ref{sec:MultiCycle}, we turn our attention to the complement of graphs known as cactus graphs. Finally, in Section \ref{sec:Future}, we discuss a possible extension of this work to the study of the zero forcing number of the complement of partial $k$-trees.  

\section{Graphs which do not contain a $K_{r,s}$ subgraph}\label{sec:KeyResult}
In this section, we consider graphs which do not contain a complete bipartite graph $K_{r,s}$ as a subgraph. We first consider the most general case, and prove a bound on the zero forcing number of the complement, in terms of $r$ and $s$. The following theorem is a main result and will be applied to prove further results throughout the paper.
% As mentioned in the introduction, one of the numerous contributions from the work in  \cite{AIM} was a determination of the zero forcing number of the complement of a tree. The method of proof given in \cite{AIM} did not rely on the rules of zero forcing, but rather by first establishing the maximum nullity for the complement of a tree then noting equality between the zero forcing number and the maximum nullity must also hold in this extreme case. 

% Part of the purpose of this work is to go back to the complement of trees and prove directly that if $T$ is a tree on $n$ vertices, then either $\operatorname{Z}(\overline{T})=n-3$ or $\operatorname{Z}(\overline{T})=n-1$, if $T$ is a star on $n$ vertices. As our analysis unfolded in this study, it became clear (\cite{MY}) that our techniques could be extended to a much broader class of graphs well beyond the case of trees. For instance trees may be considered as the class of bipartite graphs that do not contain $K_{2,2}$ as an induced subgraph.  This leads to the following subsection which contains a key result that will be used throughout our work.

% \subsection{Graphs which do not contain a $K_{r,s}$ subgraph}
% In this subsection we focus on the zero forcing number of the complement of graphs that do not contain a specific complete bipartite graph as a subgraph.

\begin{thm}\label{thm:SubKrs}
Let $G$ be a graph on $n$ vertices which does not contain $K_{r,s}$ as a subgraph for some $r+s\leq n$. Then $\operatorname{Z}(\overline{G})\geq n-r-s+1$.
\end{thm}

\begin{proof}
Suppose that $\operatorname{Z}(\overline{G})\leq n-r-s$ and let $B$ be a zero forcing set of size $n-r-s$ for $\overline{G}$. Let $W=V-B$, the set of initially white vertices and let $W=\{w_1,\dots,w_{r+s}\}$. There must exist a vertex $v_1\in B$ whose only white neighbor in $\overline{G}$ lies in $W$, or else no forcing occurs. Let $v_1$ force $w_1$. Then in $G$, $v_1$ is adjacent to $\{w_2,\dots,w_{r+s}\}$. In order for forcing to continue, there exists $v_2\in B$ whose only remaining white neighbor in $\overline{G}$ lies in $W-w_1$. Note it is possible that $v_2=w_1$, but $v_2$ cannot be $v_1$ or $\{w_2,\dots,w_{r+s}\}$. Let the vertex which $v_2$ forces be $w_2$. Then $v_2$ must be adjacent in $G$ to $\{w_3,\dots,w_{r+s}\}$ in $G$. Continuing in this manner, we can find vertices ${v_1,\dots,v_{r+s}}$ such that $v_i$ forces $w_i$ in $\overline{G}$ for $1\leq i\leq r+s$. Further, for $1\leq i\leq r+s-1$, $v_i$ is adjacent to $\{w_{i+1},\dots,w_{r+s}\}$ in $G$. Note in each case it is possible that $v_i=w_{j}$ for some $1\leq j\leq i-1$, but $v_i$ cannot be $v_k$ for $2\leq k\leq i-1$ or $w_j$ for $i\leq j\leq r+s$. Therefore the vertices $\{v_1,\dots,v_r\}$ and $\{w_{r+1},\dots,w_{r+s}\}$ contain $K_{r,s}$ as a subgraph.
\end{proof}

\begin{figure}[h!]
    \centering
    \includegraphics[scale=.25]{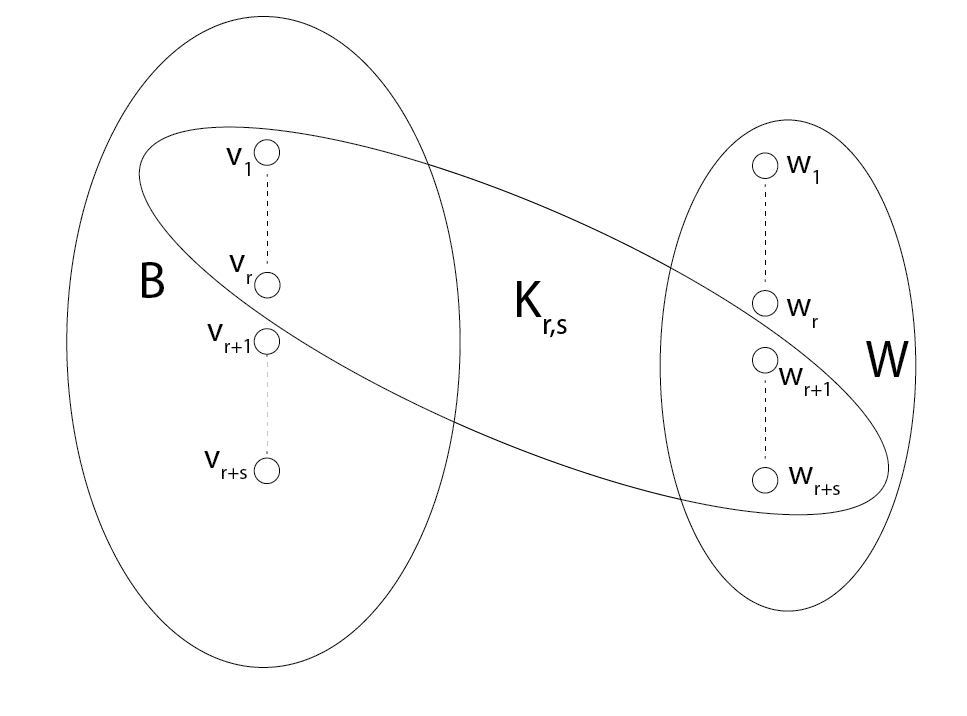}
    \caption{The graph $G$ and the $K_{r,s}$ subgraph found in the proof of Theorem \ref{thm:SubKrs}}
    \label{fig:KrsSubgraph}
\end{figure}

We will now apply Theorem \ref{thm:SubKrs} to prove results about the zero forcing number of the complement of various other graph families. First, we turn our attention to trees. Since trees do not contain any cycles, they do not contain $K_{2,2}$ as a subgraph. As mentioned in the introduction, one of the numerous contributions from the work in  \cite{AIM} was a determination of the zero forcing number of the complement of a tree. However, the method of proof given in \cite{AIM} did not rely on the rules of zero forcing. Instead, the proof relied on first establishing the maximum nullity for the complement of a tree, and then noting equality between the zero forcing number and the maximum nullity must also hold in this extreme case. We will now re-prove this result, using Theorem \ref{thm:SubKrs} and the properties of zero forcing.

\begin{prop}\label{treesandstars}
Let $T$ be a tree on $n \geq 4$ vertices which is not the star graph. Then $\operatorname{Z}(\overline{T})=n-3$. For the star graph $K_{1,n-1}$, with $n \geq 3$, $\operatorname{Z}(\overline{K_{1,n-1}})=n-1$.
\end{prop}

\begin{proof}
First observe, that for the star graph $K_{1,n-1}$, the complement $\overline{K_{1,n-1}}$ is $K_{n-1}\cup K_{1}$. Since for any $r \geq 2$ $\operatorname{Z}(K_r)=r-1$, it is clear $\operatorname{Z}(\overline{K_{1,n-1}})=n-2+1=n-1$.
 Let $T$ be a tree that is not a star. 
Since trees do not contain $K_{2,2}$, $\operatorname{Z}(\overline{T})\geq n-3$ by Theorem \ref{thm:SubKrs}. To show equality, we describe a zero forcing set $B$ for $\overline{T}$ of size $n-3$. Since $T$ is not a star is must have diameter at least 3. Let $v_0$ be a leaf in $T$. Then there exists an induced path on four distinct vertices in $T$ on the vertices $\{v_0, v_1, v_2, v_3\}$. Define $B=V(T) \setminus \{v_0, v_1, v_2\}$. Since $v_2$ is in $B$ and is adjacent to $v_0$ in  $\overline{T}$ and not adjacent to either $v_1$ nor $v_3$, it follows that $v_2$ can force $v_0$. Then $v_0$ can now force $v_3$, as $v_0$ is not adjacent to $v_1$ in $\overline{T}$. Finally, $v_1$ can be forced. Thus $B$ is a zero forcing set of size $n-3$.
\end{proof}

Noting that trees are a special class of bipartite graphs, Proposition \ref{treesandstars} can be generalized to all $K_{2,2}$-free bipartite graphs which are not the star graph. In order to prove the generalization, we require the following known lemma.

\begin{lem}\label{InducedSubgraphs}\cite{AIM}
For a graph $G$ on at least 3 vertices, $\operatorname{Z}(G)\ge |G|-2$ if and only if $G$ does not contain $P_2\cup P_2 \cup P_2$, $P_3\cup P_2$, $P_4$, $\ltimes$, or a dart as an induced subgraph. 
\end{lem}

\begin{prop}
If $G$ is a connected $K_{2,2}$-free bipartite graph with $|G| = n\geq 4$ vertices and $G\ne K_{1,n-1}$, then $\operatorname{Z}(\overline{G}) = n-3  $.
\end{prop}
\begin{proof} Observe that if $\overline{G}$ is disconnected, then $\overline{G}$ is the disjoint union of two complete graphs. This implies $G$ is the complete bipartite graph $K_{r,s}$, which contradicts our assumptions of $G$ being $K_{2,2}$-free and $G\ne K_{1,n-1}$.
Now let $G$ be a connected bipartite graph on $n\geq 4$ vertices that is $K_{2,2}$-free and assume $G\ne K_{1,r}$.  By Theorem \ref{thm:SubKrs}, we know $\operatorname{Z}(\overline{G}) \geq  n-3$. Assume $\operatorname{Z}(\overline{G}) = n-2$. From Lemma \ref{InducedSubgraphs}, $\overline{G}$ does not contain $P_4$ as an induced subgraph and therefore diam$(\overline{G}) = 2$. This implies  $\overline{G}$ is a connected cograph \cite{clsb}. However, it is well-known that any connected cograph has a disconnected complement (since any cograph is either the union or join of two smaller cographs).
%$\overline{G}$ is a join of two complete graphs, i.e. $G = K_r\vee K_s$, $r,s\in\mathbb{Z}_+$. This gives our original graph $G$ as the disjoint union of $\overline{K_r}$ and $\overline{K_s}$, contradicting our assumption of $G$ being connected. Therefore, we have $\operatorname{Z}(\overline{G}) \ne  n -2  $. Since $G$ is $K_{2,2}$ free, by Theorem \ref{thm:SubKrs}, $\operatorname{Z}(\overline{G}) \geq n-3$. Therefore, $\operatorname{Z}(\overline{G})=n-3$
\end{proof}

\section{Unicyclic Graphs}\label{sec:uni}
In Section \ref{sec:KeyResult}, Proposition \ref{treesandstars} gives us the zero forcing number for the complement of any acyclic graph. The next logical step is to examine graphs that contain precisely one cycle. Such a graph is called a unicyclic graph. 
\begin{defn}
 A unicyclic graph is a connected graph containing exactly one cycle such that every vertex on the single cycle is adjacent to at least one vertex in a single tree. 
%  See Figure \ref{fig:unicyclic}. 
 \end{defn}
 
 \begin{figure}[h!]
     \centering
     \includegraphics[height = 4cm, width=4cm]{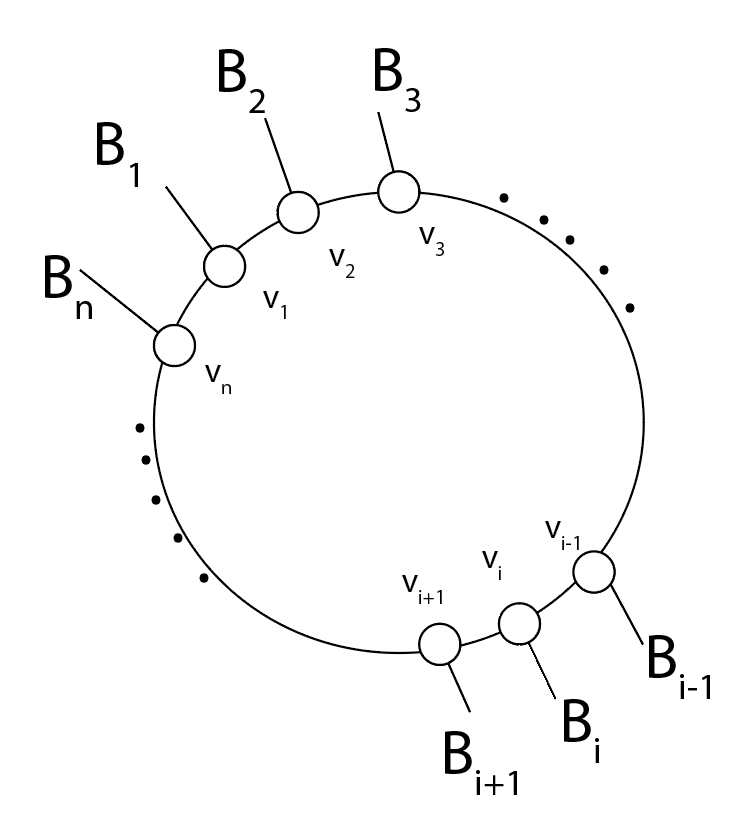}
     \caption{The general form of a unicyclic graph}
     \label{fig:unicyclic}
 \end{figure}
 
%   \begin{center}
%  \includegraphics[height = 4cm, width=4cm]{unicyclicgraphpic}
%  \end{center}
 
 Let $G$ be a unicyclic graph with vertices $v_1$, $v_2$, $\ldots$, $v_n$ on the single cycle, and forests $B_1$, $B_2$, $\ldots$, $B_n$, where each forest contains $m_i$ vertices excluding $v_i$. Note that in $G$, any vertex of $B_k$ with $k \neq i$ is nonadjacent to $v_i$ otherwise we create a cycle. A simple lower bound for the zero forcing number of the complement of this graph is therefore achieved by bounding the minimum degree of a vertex in this graph. If we let $m_{\max}$ be the maximum size of a forest in $G$ and consider the degree of $v_{\text{max}}$, a vertex on the cycle in $G$ that is adjacent to all vertices in the maximum sized forest $B_{\text{max}}$, we observe the following lower bound. 
 
%  Considering the figure above, let $G$ be a unicylic graph with vertices $v_1$, $v_2$, $\ldots$, $v_n$ on the single cycle, and trees $T_1$, $T_2$, $\ldots$, $T_n$ where each tree contains $m_i$ vertices excluding $v_i$. Note that any vertex of $T_k$ with $k \neq i$ is nonadjacent to $v_i$ otherwise we create a cycle. A simple lower bound for the zero forcing number of the complement of this graph is therefore achieved by bounding the minimum degree of a vertex in this graph. If we let $m_{\max}$ be the maximum size of a tree in $G$ and $v_{\text{max}}$ be a vertex on the cycle in $G$ that is adjacent to all vertices in the maximum sized tree $T_{\text{max}}$, we have the following observation. 
 \begin{obs}\label{ZFLowerBounndUnicylicGraph}
   Let $G$ be a unicyclic graph as labeled in Figure \ref{fig:unicyclic}. Then, \[n + \sum \limits_{j = 1}^n m_j - m_{\text{max}} - 3 \leq \operatorname{Z}(\overline{G}).\]
 \end{obs}
 Not only can this bound can be improved upon, but we can also find the exact zero forcing number of the complement of any unicyclic graph. This is the main topic for what follows in this section.

\begin{rem}\label{ZFUnicyclicBounds} If $G$ is a tree, we know $\operatorname{Z}(\overline{G}) = n-3$ or in the special case when $G$ is a star, $\operatorname{Z}(\overline{G}) = n-1$. We can view a unicyclic graph $G$ as being a tree with an additional edge, and therefore, $\overline{G}$ being the complement of a tree minus a single edge. So, combining this observation with the inequality for any graph $G$:
    \[-1 \le \operatorname{Z}(G) - \operatorname{Z}(G-e) \le 1\]
we have
    \[n-4 \le \operatorname{Z}(\overline{G}) \le n-2\]
for any unicyclic graph $G$. It is also helpful to note that the only graph on $n$ vertices whose zero forcing number is $n-1$ is the complete graph on $n$ vertices.
\end{rem}
We now characterize the unicyclic graphs $G$ with $Z(\overline{G}) = n-4, n-3, n-2$.

\begin{prop}\label{InducedP4}
    Suppose $G$ is a connected unicyclic graph, $|G|\ge 5$ with a cycle of length 4 or more. Then $\overline{G}$ contains $P_4$ as an induced subgraph. 
\end{prop}

\begin{proof}
Label the vertices of the cycle of $G$ as $v_1, ... , v_k$ with $k \geq 4$ as in Figure \ref{fig:unicyclic}. If $k \geq 5$, then the subgraph induced by the vertices $\{v_1, v_2,v_3, v_4\}$ is $P_4$. On the other hand if $k=4$, then since $|G|\geq 5$, there exists a vertex $u$ in $G$ not on the cycle but adjacent to a vertex on the cycle. Suppose, without loss of generality that $u$ is adjacent to $v_4$. Then the subgraph induced by the vertices $\{v_2,v_3, v_4,u\}$ is $P_4$. Thus in either case $G$ contains $P_4$ as an induced subgraph, hence so does $\overline{G}$.
%Now, suppose $G$ has a cycle of length four. Label the vertices of the cycle as $v_0, v_1, v_2, v_3$ similar to above. Since $|G| \ge 6$, there will be at least two other vertices of $G$ not part of the cycle, label them as $v_4, v_5$. We have three cases on the adjacency's of $v_4$ and $v_5$ which can be seen in the {\color{blue}figure below}. In all cases, $\overline{G}$ contains an induced $P_4$ by taking $\overline{G}[v_0,v_4,v_3,v_1]$.
\end{proof}

\begin{cor}\label{InducedP4Cor}
If $G$ is a unicyclic graph  of size $n\ge 5$ such that $G$ contains a cycle of length at least four, then $\operatorname{Z}(\overline{G}) \le n-3$.
\end{cor}

\begin{prop}\label{uni-n2}
    Let $G$ be a unicyclic graph on $n \geq 4$ vertices. Then, $\operatorname{Z}(\overline{G}) = n-2$ if and only if $G = K_{1,n-1}+e$ or in the special case when $n=4$, $G$ could also be isomorphic to $C_4$. 
\end{prop}

\begin{proof}
Let $G$ be a unicyclic graph. Assume $n \geq 5$ and that $G = K_{1,n-1}+e$. Then $\overline{G}$ has the structure of $(K_{n-1} - e) \sqcup \{v_0\}$ where $v_0$ is the center (or vertex of maximum degree) of $G$. So, \[\operatorname{Z}(\overline{G}) = \operatorname{Z}(K_{n-1}-e) + \operatorname{Z}(\{v\}) = (n - 3) +1 = n-2.\]

Now suppose $\operatorname{Z}(\overline{G}) = n-2$. If $G$ contains a cycle of length 4 or more and $n\geq 5$, then by Corollary \ref{InducedP4Cor}, we get $\operatorname{Z}(\overline{G}) \le n -3$. Hence $G$ must contain a cycle of length 3. Label the vertices on the 3-cycle of $G$ as $a, b, c$. Assume, that at least two vertices on this 3-cycle have degree at least 3, say deg$(a) \geq 3$ and deg$(b)\geq 3$. Let $u$ (and $v$) be a neighbour of $a$ (of $b$) not on the cycle. Now consider the subgraph of $G$ induced by the vertices ${u,a,b,v}$, which induces a $P_4$. Hence 
$\overline{G}$ contains $P_4$ as an induced subgraph. Thus $\operatorname{Z}(\overline{G}) \leq  n-3$, which is a contradiction. Hence at most one vertex on the 3-cycle has degree at least three. Since $n \geq 4$, assume vertex $a$ has degree at least 3. It is straightforward to deduce that if $G$ is not isomorphic to
$K_{1,n-1} + e$, then $G$ will contain an induced $P_4$.  Finally, if $n=4$, then the complement of $G$ is either a cycle of 4 vertices or is a path on 3 vertices with an additional isolated vertex. In either case the zero forcing number is two as needed. This completes the proof.
\end{proof}

\begin{rem}
From our previous proposition, all other unicyclic graphs have $\operatorname{Z}(\overline{G}) = n-3$ or $n-4$. In fact, those unicyclic graphs with $\operatorname{Z}(\overline{G}) = n-4$ are contained among those with $C_4$ as an induced subgraph, as the next proposition provides. 
\end{rem}
\begin{prop}\label{uni-c4}
    Let $G$ be a unicyclic graph and $\operatorname{Z}(\overline{G}) = n-4$. Then, $C_4$ is an induced subgraph of $G$.
\end{prop}

\begin{proof}
Let $G$ be unicyclic and $\operatorname{Z}(\overline{G}) = n-4$. If $G$ contains a cycle $C_j$ for $j > 4$, then $G$ is $K_{2,2}$-free, and so $\operatorname{Z}(\overline{G})\ge n-3$, by Theorem \ref{thm:SubKrs} contradicting our assumption of the zero forcing number. So, $G$ must contain the cycle $C_4$ as an induced subgraph.  
\end{proof}

From the proof of the previous proposition we observe that if $G$ is a unicyclic graph that contains $C_{\ell}$ with $\ell \geq 5$ as an induced subgraph, then $\operatorname{Z}(\overline{G}) = n-3$. Before we state the next result, we note that if $n=3, 4,$ then it is straightforward to verify the possible list of zero forcing numbers of the complements of such unicyclic graphs. 

\begin{thm} \label{mega_n-4}
Suppose $G$ is a unicyclic graph on $n \geq 5$ vertices. Assume $G$ contains $C_4$ as an induced subgraph. 
Then the following statements hold:

\begin{enumerate}
\item If at most one vertex on the cycle $C_4$ has degree 2, then $\operatorname{Z}(\overline{G}) = n-4$;
\item If exactly two vertices on the cycle $C_4$ have degree 2, and if
\begin{enumerate}
\item these vertices are adjacent, then $\operatorname{Z}(\overline{G}) = n-4$, or
\item  these vertices are not adjacent, then $\operatorname{Z}(\overline{G}) = n-3$;
\end{enumerate}
\item If exactly three vertices on the cycle $C_4$ have degree 2, then $\operatorname{Z}(\overline{G}) = n-3$.
\end{enumerate}
\end{thm}
\begin{proof}
Suppose the vertices on the cycle are labeled consecutively as $\{a,b,c,d\}$. 

Consider $G$ as in (1). First assume that no vertex on the cycle $C_4$ has degree two and let $\{\alpha, \beta, \gamma, \delta\}$ be neighbours of  $\{a,b,c,d\}$, respectively not on the cycle (see Figure \ref{fig1}). 
\begin{figure}[!h]
	\centering
	\includegraphics[scale=.8]{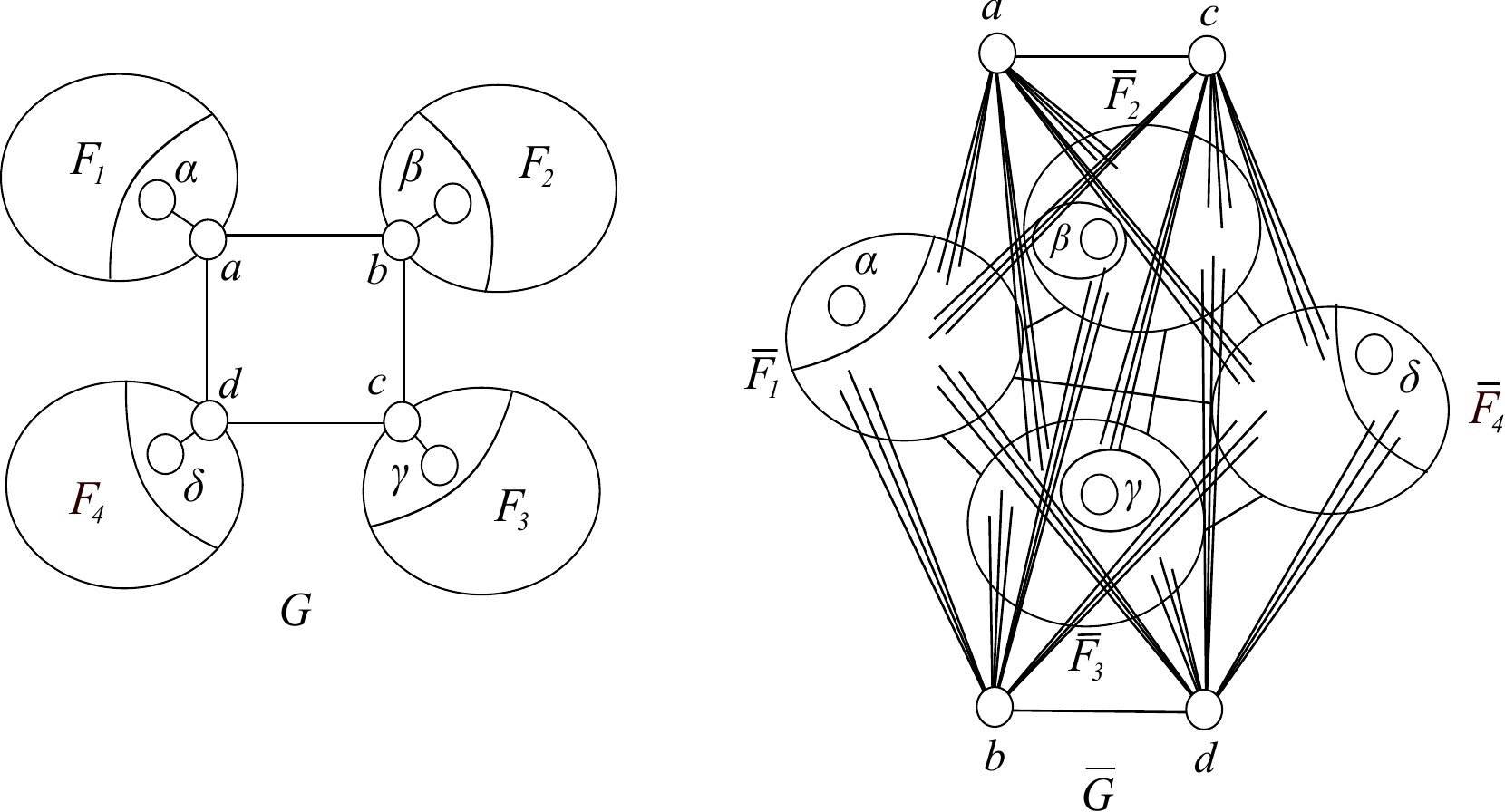} 
	\caption{Unicyclic graph containing $C_4$} \label{fig1}
\end{figure}
Let $Z = V(\overline{G}) \setminus \{\alpha, b,c,d\}$. We claim that $Z$ is a zero forcing set for $\overline{G}$. By definition of $Z$ if follows that vertex $a$ in $\overline{G}$ has only one white neighbour, namely $c$. So $a$ can force vertex $c$. From here we note that vertex $c$ now has only one white neighbour, namely $\alpha$. So $c$ can force $\alpha$. At this stage we observe that $\beta$ has only one white neighbour, namely $d$, since $\beta$ is adjacent to $b$ in $G$. Thus $\beta$ can force $d$, and finally the last white vertex ($b$) can be forced. Hence $\operatorname{Z}(\overline{G})  \leq n-4$, but from the remarks at the beginning of this section we also know  $\operatorname{Z}(\overline{G})  \geq n-4$, from which we conclude that  $\operatorname{Z}(\overline{G})  = n-4$. If only one vertex has degree two on the cycle $C_4$, then a very similar argument can be applied as used above to verify that  $\operatorname{Z}(\overline{G})  = n-4$. In this case, assume that vertex $c$ has degree 2, and that vertices $\{\alpha, \delta\}$ are neighbours of vertices $\{a,d\}$, respectively not on the cycle. Now let $\operatorname{Z} = V(\overline{G}) \setminus \{\alpha, \delta, b,d\}$. As above it follows that $Z$ is a minimum zero forcing set for $\overline{G}$.

Suppose $G$ satisfies (2a). Assume the two adjacent vertices on the cycle with degree two are $c$ and $d$. Using a similar approach as in the previous case, suppose $\alpha$ and $\beta$ are neighbours of $a$ and $b$, respectively not on the cycle. Let $Z = V(\overline{G}) \setminus \{\alpha, \beta,b,d\}$. We claim that $Z$ is a zero forcing set for $\overline{G}$. Since $a$ is not adjacent to any of $\alpha, b,$ or $d$ in  $\overline{G}$ it follows that $a$ can force $\beta$. From here we have the following possible sequence of forces: $c$ can force $\alpha$; $\beta$ can force $\alpha$, and finally $b$ can forces $d$. Hence $Z$ is  zero forcing set for  $\overline{G}$, from which it follows that $\operatorname{Z}(\overline{G})  = n-4$ as $n-4$ is a lower bound as observed above.

Regarding (2b), assume that the two non-adjacent vertices on the cycle with degree two are $b$ and $d$, and let $\alpha$ and $\gamma$ be neighbours of $a$ and $c$, respectively not on the cycle. Denote the forests adjacent to $a$ and $c$ by $F_1$ and $F_3$, respectively (see Figure \ref{fig2}).
\begin{figure}[!h]
	\centering
	\includegraphics[scale=.8]{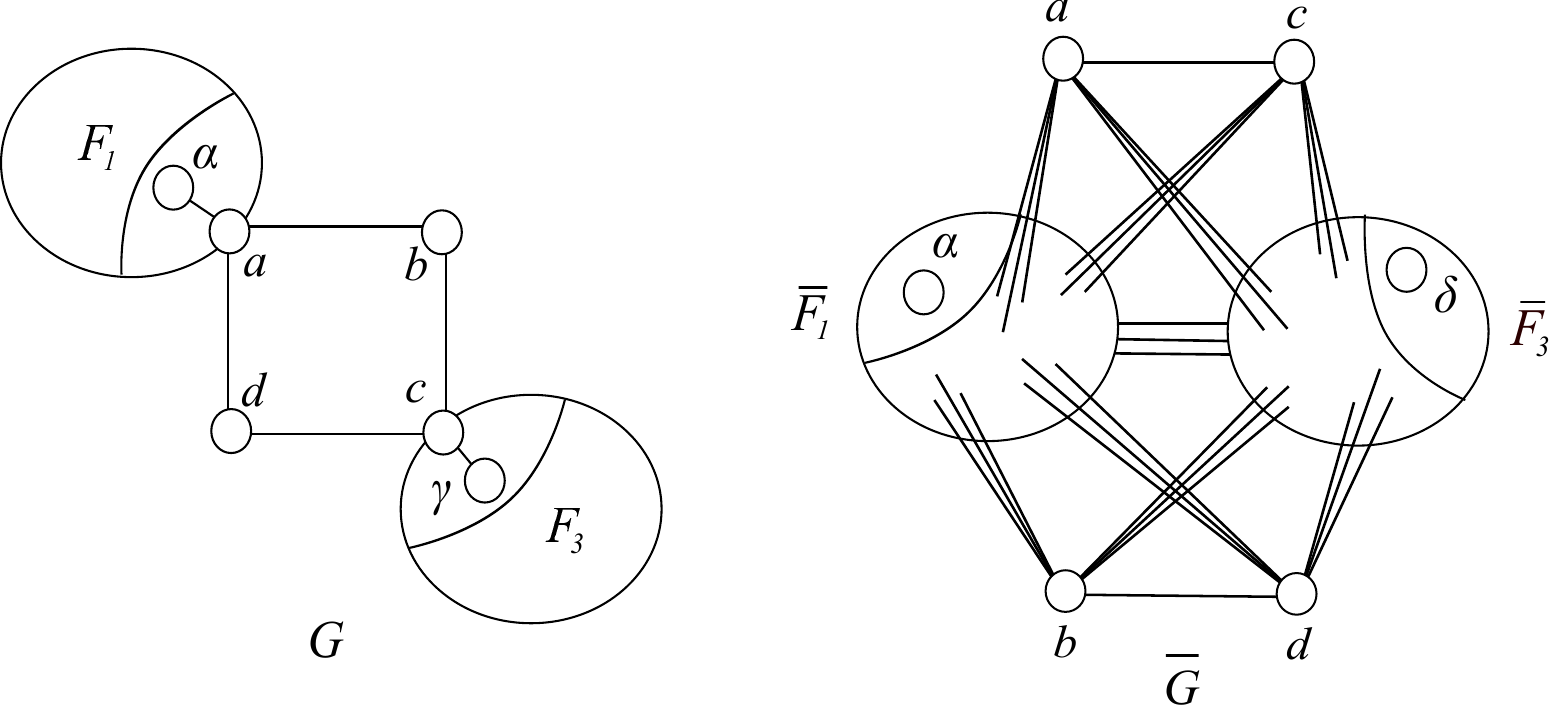} 
	\caption{Unicyclic graph containing $C_4$ satisfying item (2b)} \label{fig2}
\end{figure}
To prove that the zero forcing number of $\overline{G}$ is   $n-3$, we will establish that no subset of vertices with size $n-4$ can be  a zero forcing set for  $\overline{G}$. Let $B \subset V(\overline{G})$ with $|B|=n-4$ and consider the size of the set $(V(F_1) \cup V(F_3)) \cap B$. If $B=V(F_1) \cup V(F_3)$, then it is clear that no forces are possible. Now assume $|(V(F_1) \cup V(F_3)) \setminus B|=1$. Thus exactly one vertex on the cycle of $G$ is blue, which leads to two possible cases: either $a$ is blue or $b$ is blue. In either case neither $a$ nor $b$ can perform an initial force. If there is a vertex $x$ that is blue and in $V(F_1) \cup V(F_3)$, then $x$ has at least two white neighbours among the vertices on the cycle in either case. Hence $B$ is not a zero forcing set for $\overline{G}$. Next assume that $|(V(F_1) \cup V(F_3)) \setminus B|=2$. In this case there are exactly two vertices on the cycle that are blue, which leads to three possible situations: $(i)$ $a$ and $c$ are blue; $(ii)$ $b$ and $d$ are blue; or $(iii)$ $a$ and $b$ are blue. 

First consider the case in which $a$ and $c$ are blue. Then it follows that no vertex in   $V(F_1) \cup V(F_3)$ can apply a force (as both $b$ and $d$ are white). So assume that $a$ can force a vertex in   $V(F_1) \cup V(F_3)$, call it $x$. This can only happen if $x$ is in $F_1$, is a neighbour of $a$,  and both $x$ and $\alpha$ are white, and they are the only white vertices in $V(F_1) \cup V(F_3)$. Then $c$ can force $\alpha$. However, no other forces are possible as both $b$ and $d$ are white. So now assume that $b$ and $d$ are blue. Then since $\{b, d\}$ are joined to the vertices in  $V(F_1) \cup V(F_3)$ in  $\overline{G}$, neither $b$ nor $d$ can perform a force at this initial stage. Assume there is a vertex $x$, in $F_1$ that performs the first force. Then it follows that the remaining two white vertices ($u$ and $v$) must also be in $F_3$ (since the vertices of $F_1$ are joined to the vertices in $F_3$),  that $x=\alpha$ and $x$ must force $c$. Since $c$ has at least two white neighbours (including $a$), $c$ cannot perform a force. So there must be another vertex $y$ that can perform a force at this stage. Since no blue vertex in $F_3$ can force, it follows that $y$ is in $F_1$ and $y$ is not $\alpha$. In this case the only possible force is for $y$ to force vertex $a$. In this case we must have that both $\alpha$ and $y$ adjacent $u$ and $v$ in $G$. Hence $G$ contains two 4-cycles, which is a contradiction. The last case $(iii)$ follows in a similar manner. 

Now assume $|(V(F_1) \cup V(F_3)) \setminus B|=3$. In this case there are two cases to consider since there is exactly one white vertex on the cycle. Suppose either $a$ is white or $b$ is white. In either case, we can apply arguments similar to the above case to deduce that $B$ was a zero forcing set, then $G$ would contain more than one 4-cycle, which is not possible. Finally, suppose  $|(V(F_1) \cup V(F_3)) \setminus B|=4$. Then all white vertices are among the vertices of $V(F_1) \cup V(F_2)$. Therefore  either none of the vertices $\{a,b,c,d\}$ can perform a force or, say all of the vertices in $F_1$ are adjacent to $a$, In the latter case, $a$ could force a vertex in $F_3$, only if there is just one white vertex in $F_3$. After this no further forces are possible. Additionally, at least one of $F_1$ or $F_3$ must contain at most one white vertex, since the vertices of $F_1$ are joined to the vertices of $F_3$ in $\overline{G}$. Suppose $F_1$ contains at most one white vertex, say $u$, and $x,y,z$ are three white vertices in $F_3$. Then the first force is accomplished by a vertex $w$ in $F_3$, and in this case $w$  forces $u$. Thus $w$ must be adjacent to all of $x,y,z$ in $G$. The next possible force is by, say $w'$ and $w'$ can force $x$. Hence $w'$ must be adjacent to both $y$ and $z$ in $G$. Thus the vertices $w,y,z,w'$ form a 
 4-cycle in $G$, which is not possible. Consequently, we have argued that there is not a zero forcing set for $\overline{G}$ of size $n-4$, so we may conclude that $\operatorname{Z}(\overline{G})  = n-3$. 

Finally, suppose (3) holds. Then we can argue is a manner similar to (2b) whereby we consider all possible  subsets $B$ of size $n-4$ and verify that none of them can be a minimum zero forcing set by breaking the argument depending on the overlap between $B$ and the vertices in the forest adjacent to exactly one vertex on the cycle. 

This completes the proof.\end{proof}

\begin{cor}\label{z-uni}
All unicylic graphs $G$ with $G\ne K_{1,n-1}+ e$ and containing  $C_\ell$, $\ell \geq 5$ as an induced subgraph, or satisfy (2b) or (3) in Theorem \ref{mega_n-4} have $\operatorname{Z}(\overline{G}) = n-3$. 
\end{cor}

It is evident from the analysis above (see, for example, Corollary \ref{z-uni}) that for `most' unicyclic graphs $U$ on $n$ vertices we have $\operatorname{Z}(\overline{U}) = n - 3$. We know from the work above that there is essentially only one family of unicyclic graphs (on $n$ vertices) with the property that the zero forcing number of the complement of such graphs is $n-2$. Further it follows , that if $n \geq 5$, $\operatorname{Z}(\overline{C_n}) = n - 3$ while $\operatorname{Z}(\overline{C_3}) = 3$ and $\operatorname{Z}(\overline{C_4}) = 2$. Another illustrative example along these lines is the the $n-$sunlet graph. An $n-$sunlet graph is the graph on $2n$ vertices obtained by attaching $n$ pendant edges to a cycle graph $C_n$, i.e. the coronas $C_n \bigodot K_1$.  
Notice, the $n-$sunlet graph can be viewed as a $(n,n)-$sunlet graph, and thus, we obtain the following corollary.
If $n \geq 3$ and $n\neq 4$, then $\operatorname{Z}(\overline{C_n \bigodot K_1}) = 2n - 3$; if $n=4$, then $\operatorname{Z}(\overline{C_n \bigodot K_1}) = 2n - 4$. 

We close this section with an application of the work above to address the following question: Which connected unicyclic graphs $U$ satisfy $\operatorname{Z}(U)=\operatorname{Z}(\overline{U})$. This was partially addressed in the paper \cite{EKY15} where they characterized the connected unicyclic graphs $U$ on $n$ vertices with $\operatorname{Z}(U)=n-3$. 

\begin{cor}
Suppose $U$ is a connected unicyclic graph on $n$ vertices with $n \geq 4$. Then $\operatorname{Z}(U)=\operatorname{Z}(\overline{U})$ if and only if 
\begin{enumerate}
\item $n=4$ and $U$ is either $C_4$ of $K_3$ with a pendant added;
\item $n=5$ and $U$ is either $C_4$ with a pendant added or $K_3$ with a single pendant vertex added to two vertices on the triangle; or
\item $n$ is at least 6, and both $U$ and $U$ complement are connected and satisfy $\operatorname{Z}(U)=\operatorname{Z}(\overline{U})=n-3$ or $n=6$ and $U$ is $C_4$ with two pendant vertices added to neighbouring vertices on the 4-cycle. 
\end{enumerate}
\end{cor}

\begin{proof}
For the cases when $n=4$ or $5$, it is simple to exhaustively check all such cases and deduce the listed claims. Suppose $n \geq 6$. First it is easy to observe that if $U$ is $C_4$ with two pendant vertices added to neighbouring vertices on the 4-cycle, then $\operatorname{Z}(U)=\operatorname{Z}(\overline{U})=2=n-4$. So now we assume that $n\geq 6$, and that $U$ is a connected unicylic graph. 

Before we proceed with the remainder of the proof, we make the following simple observation. Suppose $v$ is a vertex on the cycle of $U$ and $u$ is a neighbour of $v$ not on the cycle. Then $u$ belongs to a tree, say $T$, that is 'adjacent' to the cycle in $U$. It is easy to see that we can colour all but one vertex in $T$ and by applying the colour change rule we can force all of the vertices in $T$ and the vertex $v$. To see this note that if $T$ is a star with centre $u$, then colour all of the vertices in $T$ except $u$, else $T$ is not a star in which case we can colour $u$ and all but one vertex in the remainder of $T$. 

Using this observation, it is easy to deduce that for any unicyclic graph $U$ with girth at least five, satisfies $\operatorname{Z}(U) \leq n-5$. Hence if we have $\operatorname{Z}(U)=\operatorname{Z}(\overline{U})$, then $U$ must have girth equal to 3 or 4. If the girth of $U$ is 3, then 
by Proposition \ref{uni-c4}, it follows that $\operatorname{Z}(\overline{U}) \in \{n-2,n-3\}$. However, since $n\geq 6$ it follows that $\operatorname{Z}(\overline{U}) \neq n-2$, using Proposition \ref{uni-n2}. Thus $U$ is a connected unicylic graph that satisfies $\operatorname{Z}(U)=\operatorname{Z}(\overline{U})=n-3$. Hence by Theorem 5.2 and Corollary 5.3 of \cite{EKY15} both $U$ and $\overline{U}$ are both connected, and $U$ is the vertex sum of $K_3$ and a leaf of a star. 

Now assume that girth of $U$ is equal to 4. Then it follows by applying \cite[Thm 5.2]{EKY15} that $\operatorname{Z}(U)=\operatorname{Z}(\overline{U})=n-4$. Now using Theorem \ref{mega_n-4} we now that $U$ falls into one of the following three cases:  at most one vertex on the cycle $C_4$ has degree 2, or exactly two adjacent vertices on the cycle $C_4$ have degree 2. Applying the observation above, by colouring all but one vertex in the forests attached to the vertices of degree larger than 2 will show that in all three cases we have $\operatorname{Z}(U) \leq n-5$, which is a contradiction.
\end{proof}

\section{Cactus Graphs}\label{sec:MultiCycle}
The next natural progression of our study of the zero forcing number of graph complements is examining graphs with more than one cycle. In this section, we determine the zero forcing number for the complement of cactus graphs, which represent a structural extension of the unicyclic graphs we explored in Section \ref{sec:uni}. 

\begin{defn}
 A \emph{cactus} (sometimes called a \emph{cactus tree}) is a connected graph in which any two simple cycles have at most one vertex in common. Equivalently, a cactus is a connected graph in which every edge belongs to at most one simple cycle, or (for nontrivial cactus) in which every block (maximal subgraph without a cut-vertex) is an edge or a cycle.
 \end{defn}

 \begin{figure}[!h]
 \centering
 \includegraphics[height = 3.3cm, width=5cm]{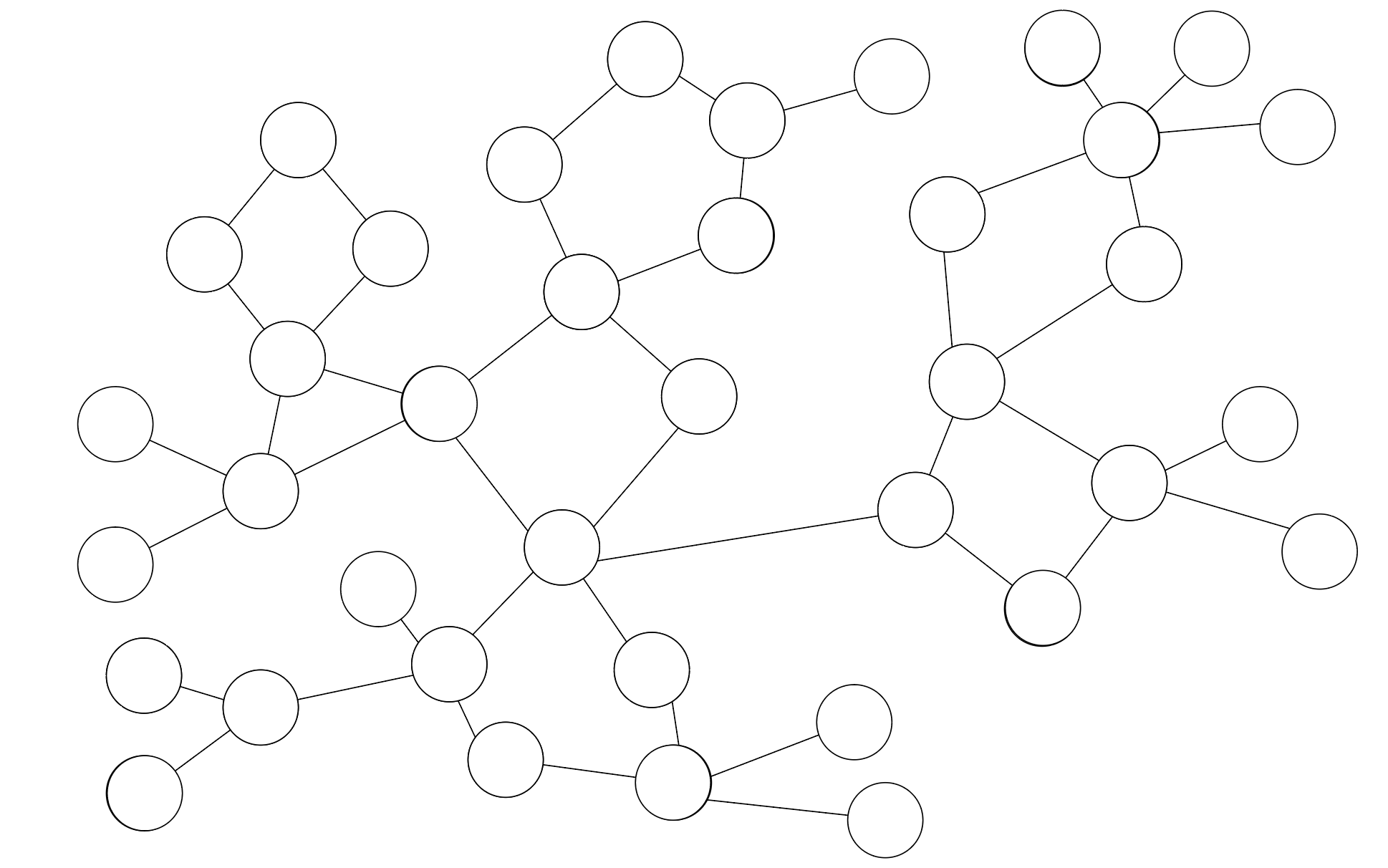}
  \caption{A Cactus Graph}
  \end{figure}
 
 \begin{obs}\label{charofcactugps}
If a graph $G$ is a cactus graph, then the following must hold:
\begin{itemize}
    \item[a.] If two adjacent vertices $v$ and $w$ each belonging to a cycle $C$ and $C'$ respectively in $G$, then no other two vertices $x,y \neq v,w$ such that $x$ is on $C$ and $y$ is on $C'$ can be adjacent. 
    \item[b.] If $v$ is on a cycle $C$ in $G$, then $v$ is adjacent to at most one vertex on all other cycles in $G$. 
    \item[c.] If a single cycle, $C$, contains a vertex $v$ adjacent to a vertex in a single tree $T$ then $v$ does not share neighbors in $T$ with any other vertex on $C$.
    \item[d.] If a single cycle $C$ contains a vertex $v$ adjacent to a vertex in a single tree $T$, then $v$ has no other neighbors in $T$,
\end{itemize}
\end{obs}

\begin{proof}
Statements a through c are true because otherwise we would create another cycle that shares an edge with an existing cycle. Statement d is true because otherwise we would construct a cycle within a tree. 
\end{proof}
 
For the remainder of this section we assume that all cacti include at least two cycles, otherwise we may refer to the results obtained in the previous sections.

We begin with a consequence of Theorem \ref{thm:SubKrs} when applied to such cacti graphs.

\begin{lem}\label{cacti-lower}
    Suppose $G$ is a cactus graph on $n$ vertices with $n \geq 4$. Then
     $\operatorname{Z}(\overline{G}) \geq n - 4$. Further if $G$ does not contain a $C_4$, then $\operatorname{Z}(\overline{G}) \geq n - 3$.
\end{lem}

\begin{proof}
    The first inequality follows directly from Theorem \ref{thm:SubKrs} since any cactus graph cannot contain $K_{2,3}$ as a subgraph. The second inequality then follows from Theorem \ref{thm:SubKrs} since it is assumed that such cacti do not contain $K_{2,2}$.
\end{proof}

Using similar ideas to that in Section 3, we need to separate the cases when $G$ may or may not contain a $C_4$. Let us first consider the cactus graphs which do not contain a $C_4$.

\begin{figure}[!h]
    \centering
    \includegraphics[scale=.8]{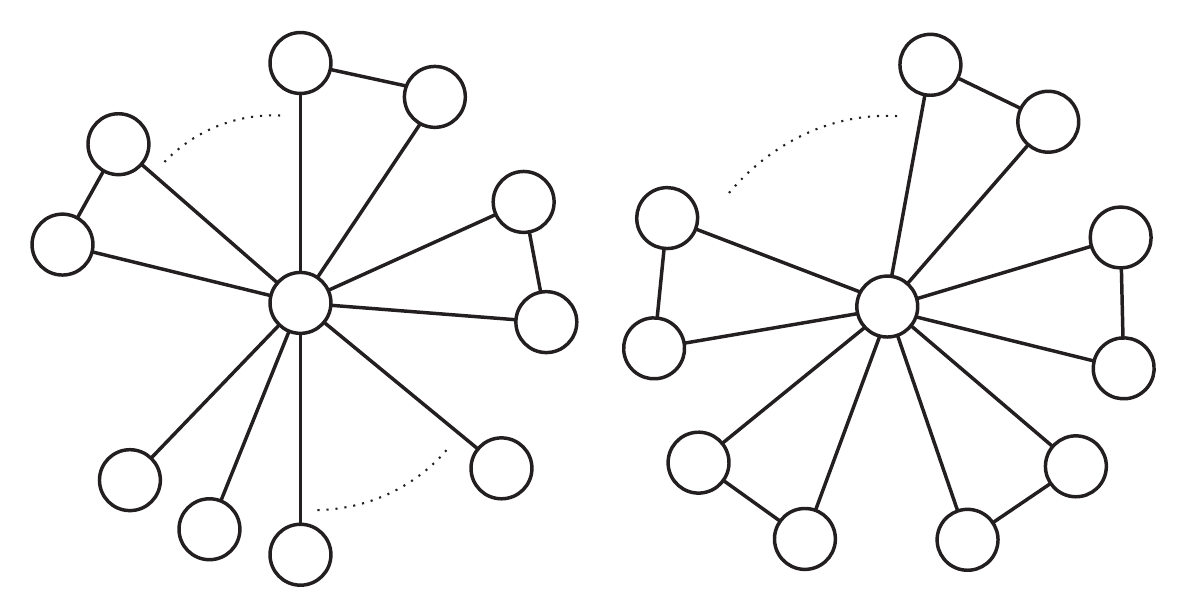}
    \caption{The star graph with at least two edges added between non-consecutive pairs of leaves.}
    \label{fig:book}
\end{figure}

\begin{prop}
 Suppose $G$ is a cactus graph (with at least two cycles) on $n$ vertices that does not contain  a $C_4$. Then,
 $\operatorname{Z}(\overline{G}) = n - 3$ or $G$ is one of the graphs in Figure \ref{fig:book} and $\operatorname{Z}(\overline{G}) = n - 2$ .
 \end{prop}

 \begin{proof} First, suppose that $G$ has at least one vertex of degree 2 incident to a cycle in $G$ and let $z$ be such a vertex. Observe that $z$ has degree $n-3$ in $\overline{G}$. Consider a zero forcing set $B$ which contains $z$ and all but one of its neighbor in $\overline{G}$. The neighbor of $z$ which is not in $B$ will necessarily be forced by $z$. Let the two remaining vertices be $x$ and $y$ and observe that they are the neighbors of $z$ on a cycle in $G$. 

Suppose the cycle containing $z$, $x$ and $y$ has length five or more. Then there exists a vertex $v$ which is adjacent to $x$ and not adjacent to $y$ in $G$ (or vice-versa). In either case $v$ belongs to $B$ and can force $y$ (or $x$). From which it follows that  $\operatorname{Z}(\overline{G})\leq n-3$. Using Lemma \ref{cacti-lower} we have that  $\operatorname{Z}(\overline{G})= n-3$.

On the other hand if the cycle containing  $z$, $x$ and $y$ is a triangle, then we have two possible situations arise. The first is that  $\overline{G}$ is connected. In this case we can apply a similar argument as above and conclude that  $\operatorname{Z}(\overline{G})= n-3$. Otherwise, $\overline{G}$ is disconnected. Since we have assumed that $G$ contains at least two cycles, it follows that $\overline{G}$ is disconnected only if $G$ contains a vertex of degree $n-1$ (which must be either $x$ or $y$). In this case $G$ is one of the graphs in Figure \ref{fig:book} and it is straightforward to check that $\operatorname{Z}(\overline{G})= n-2$ in this case. 

Next, suppose there exists no vertex of degree 2 incident to a cycle in $G$. This implies that all vertices on a cycle $C$ in $G$ have at least one neighbor that is not $C$. Using Lemma \ref{cacti-lower} it is sufficient to exhibit a zero forcing set of size $n-3$. Following the same line of reasoning from above, assume $w$ is a vertex on $C$ adjacent to vertices $x$ and $y$ both on $C$. Assume that $w'$ and $x'$ are neighbors of $w$ and $x$ not on $C$, respectively. Now consider the subset $B=V \setminus \{x,y,x'\}$ initially colored in $\overline{G}$. Then, in $\overline{G}$, it follows that $w$ can force $x'$, followed by $x'$ forcing $y$, and then $w'$ can force $x$. Hence $B$ is a zero forcing set for $\overline{G}$ and $\operatorname{Z}(\overline{G})= n-3$. 
\end{proof}

The next result presents a general upper bound on $Z(\overline{G})$ when the cactus graph $G$ contains a $C_4$.

\begin{prop}\label{C4noadjvertswithneigbhorsupperbound}
If $G$ is a cactus graph on $n$ vertices containing a $C_4$, then $Z(\overline{G}) \leq n - 3$. 
\end{prop}
\begin{proof}
Let $C_4$ be a cycle with vertices $v$, $x$, $w$, and $y$ with edges $vx$, $xy$, $yw$, and $wv$. Note that every vertex $z$ in $G$ such that $z \neq v, x, w, y$ is adjacent to at most one vertex on $C_4$ by Observation \ref{charofcactugps}. So, in $\overline{G}$, all vertices $z \neq v, x, w, y$ are adjacent to three or all of $v$, $x$, $w$, and $y$ while the adjacencies of the vertices formally on $C_4$ are $vy$ and $xw$. Note that the vertices in $\overline{G}$ can be partitioned into sets $H$, $H_y$, $H_w$, $H_x$, and $H_v$ where $H$ consists of the vertices that are adjacent to all $v$, $x$, $w$, and $y$, $H_y$ consists of all adjacent to $v$, $x$, $w$, $H_w$ consists of all adjacent to $v$, $x$, $y$, $H_x$ consists of all adjacent to $v$, $w$, $y$, and $H_v$ consists of all adjacent to $x$, $w$, $y$. Without loss of generality, leave one neighbor of $x$ white and color $x$ along with all other of its neighbors blue. Also, leave $v$ and $y$ white as well. The white neighbor of $x$ is immediately forced and all that remains white is $v$ and $y$. Any vertex in set $H_y$ will force $v$ and any vertex in $H_v$ will force $y$. Any neighbor will force the remaining white vertex. Note that at least one of $H_y$ or $H_v$ is nonempty otherwise we can use another vertex on the $C_4$ in $G$ in the place of $x$. If no vertices on the $C_4$ in $G$ have neighbors not on the $C_4$, then $G = C_4$. Either way, $\operatorname{Z}(\overline{G})\le n-3$. 
\end{proof}

%We start our analysis on the lower bound for $Z(\overline{G})$  by splitting up the cactus graphs that contain a $C_4$.

To analyze the lower bound of $Z(\overline{G})$ for cactus graphs $G$ containing a $C_4$, we focus on examining adjacencies of those vertices on $C_4$.

\begin{prop}\label{C4noadjvertswithneigbhorslowerbound}
If $G$ is a cactus graph on $n$ vertices containing a $C_4$ such that no two adjacent vertices on any $C_4$ have a neighbor not on that $C_4$, then $n - 3 \leq Z(\overline{G})$.
\end{prop}
\begin{proof}
Assume there is a zero forcing set of size $n - 4$, say $Z$. Such a set $Z$ of $n - 4$ colored vertices leaves four vertices $u_1$, $u_2$, $u_3$, $u_4$ of $G$ uncolored. If $Z$ is a zero forcing set, then there must be a vertex $z_1$ so that $z_1$ is adjacent to exactly one of the white vertices, say $u_1$ otherwise no forcing can be done in the graph. So, $z_1$ will force $u_1$ as in the Figure \ref{fig:4.6Step1}. 

\begin{figure}[!h]
    \centering
\includegraphics[height = 4cm, width = 8cm]{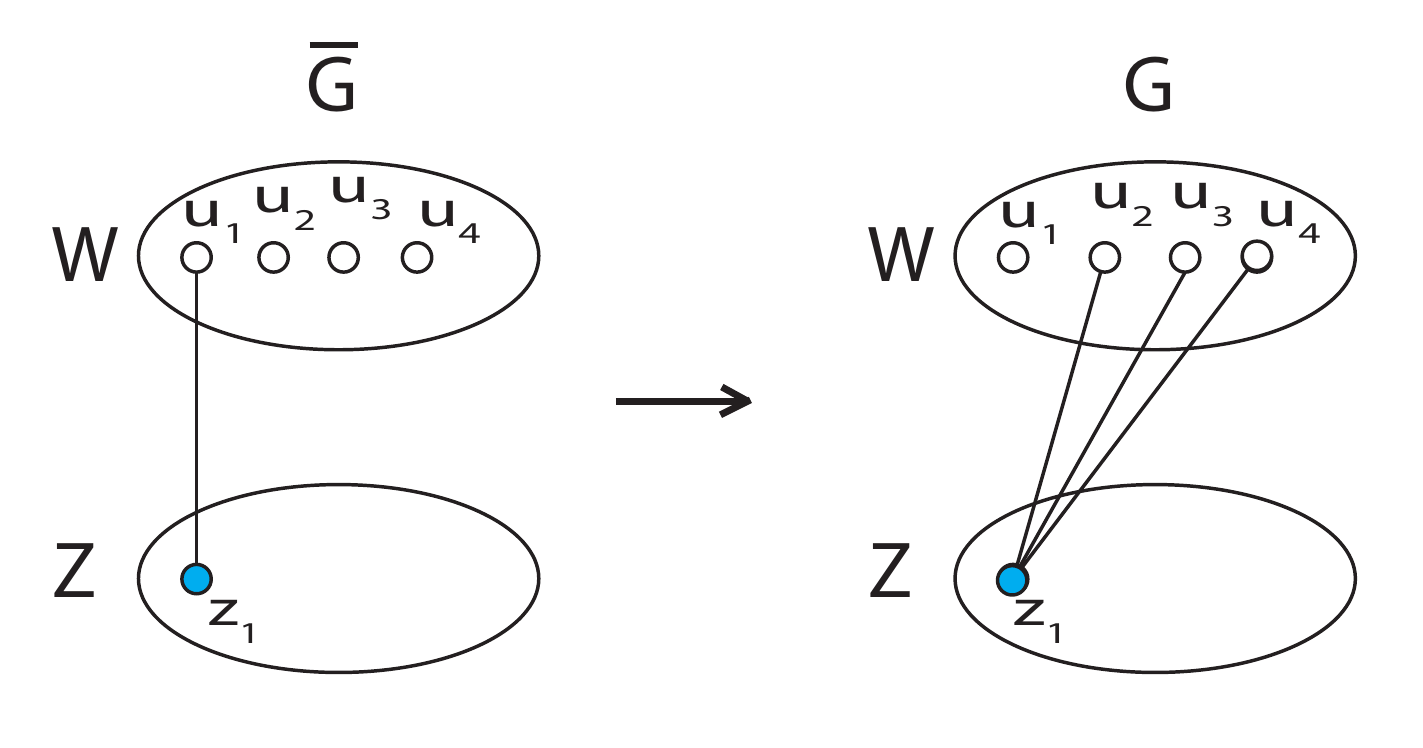}
    \caption{$z_1$ forces $u_1$ in $G$}
    \label{fig:4.6Step1}
\end{figure}

At the next stage, because $z_1$ can only force one vertex. There must be another vertex $z_2$ in the graph so that $z_2$ forces another white vertex, say $u_2$. 
There are two cases to consider. 

Case 1: $z_2 \in Z\setminus \{z_1\}$.
In this case we have the situation as described in Figure \ref{fig:4.6Case1}. 

\begin{figure}[!h]
    \centering
\includegraphics[height = 5cm, width = 8cm]{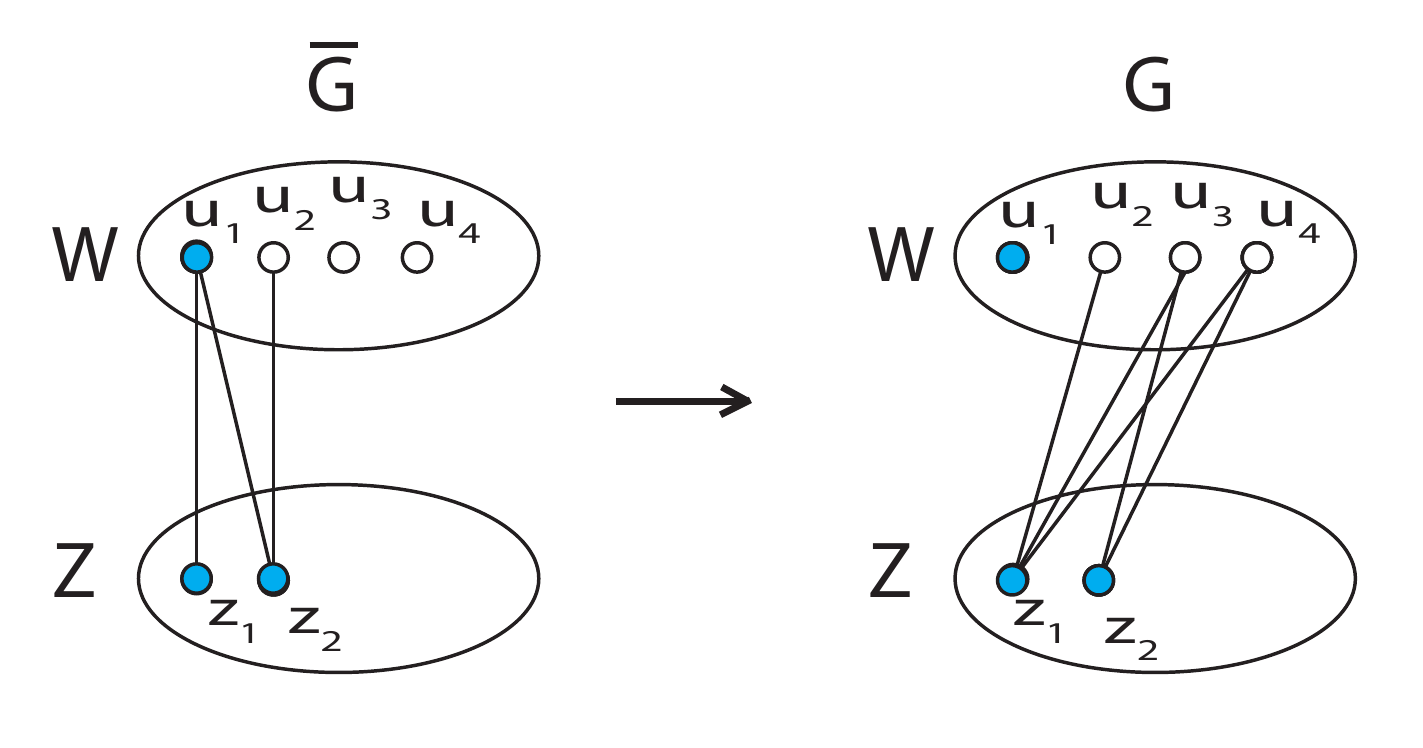}
\includegraphics[height = 5cm, width = 5cm]{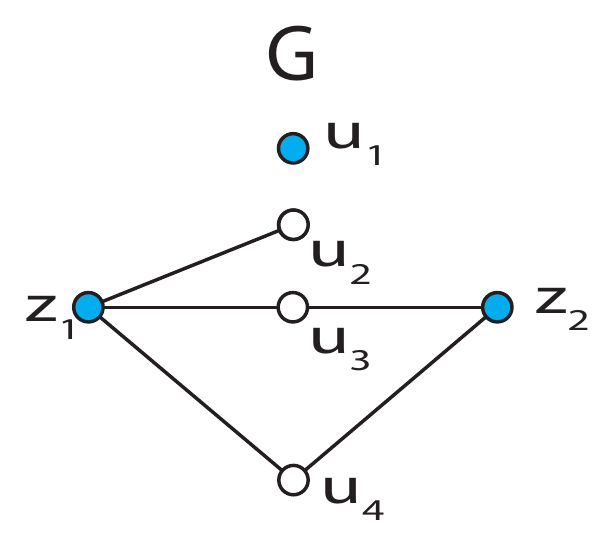}
    \caption{Case 1}
    \label{fig:4.6Case1}
\end{figure}

There must exists a third vertex $z_3$ that performs a force. Observe that $z_3$ cannot be $u_1$ nor $u_2$ as this contradicts our assumption that no two adjacent vertices on a $C_4$ can have a neighbour not on $C_4$. Hence $z_3 \in Z\setminus \{z_1, z_2\}$. In this case if $z_3$ forces either $u_3$ or $u_4$ this leads to a contradiction as the four cycle that contains $z_1, u_3, u_4$, and $z_2$ has adjacent vertices (namely $z_1$ and $u_4$) with neighbors not on that four cycle.

Case 2: $z_2$ is in $W$. In this case it must be that $z_2=w_1$. Assume that $w_1$ force, say $w_2$ (the picture is similar to Figure \ref{fig:4.6Case1}). As in Case 1, when we consider the third force, we distinguish between $z_3=w_2$ or not. If $z_3=w_2$, then we arrive at a contradiction as $G$ was assumed to be cactus. Otherwise, $z_3 \in Z\setminus \{z_1\}$ which leads to a contradiction of our assumption about all four cycles in $G$. %Therefore, we must have $n - 3 \leq Z(\overline{G})$.
\end{proof}

\begin{cor}
If $G$ is a cactus graph on $n$ vertices containing a $C_4$ such that no two adjacent vertices on any $C_4$ each have a neighbor not on $C_4$, then $Z(\overline{G}) = n - 3$. 
\end{cor}
\begin{proof}
By Propositions \ref{C4noadjvertswithneigbhorslowerbound} and \ref{C4noadjvertswithneigbhorsupperbound}, the result follows. 
\end{proof}

\begin{prop}\label{C4adjneightbors}
If $G$ is a cactus graph on $n$ vertices containing a $C_4$ such that at least two adjacent vertices on some $C_4$ each have a neighbor not on $C_4$, then $Z(\overline{G})\leq n - 4$.
\end{prop}
\begin{proof}
Let $C_4$ be a cycle with vertices $v$, $x$, $w$, and $y$ with edges $vx$, $xy$, $yw$, and $wv$. By assumption assume that $v$ is also adjacent to a vertex $a$ (not on this cycle), and let $x$ be adjacent to a vertex $b$ (not on this cycle). It is sufficient to exhibit a zero forcing set of size $n-4$ for $\overline{G}$. Consider the set $Z = V\setminus \{a,b,v,y\}$. In $\overline{G}$ the only white neighbor of $x$ is $a$. Hence $x$ can force $a$ in $\overline{G}$. Next observe that the only white neighbor of $w$ in $\overline{G}$ is $b$, so $w$ can force $b$. Then it follows that the only white neighbor of $a$ (which is now blue) is $y$. Hence $a$ can force $y$. Finally, $y$ (now blue) can force the last white vertex $v$. Therefore $Z$ is a zero forcing set for $\overline{G}$ and hence $Z(\overline{G})\leq n - 4$.
\end{proof}

\begin{cor}
If $G$ is a cactus graph on $n$ vertices containing a $C_4$ such that at least two adjacent vertices on some $C_4$ each have a neighbor not on $C_4$, then $Z(\overline{G}) = n - 4$.
\end{cor}
\begin{proof}
By Proposition \ref{C4adjneightbors} and Lemma \ref{cacti-lower} the result follows. 
\end{proof}

We close this section with a concluding summary statement concerning the zero forcing number of the complement of a cactus graph that follows from the results presented in this section.

\begin{thm}
    Suppose $G$ is a cactus graph on $n$ vertices with $n \geq 4$ and at least two cycles. Then
     $n-4 \leq \operatorname{Z}(\overline{G}) \leq n - 2$. Furthermore, 
     \begin{enumerate}
         \item $\operatorname{Z}(\overline{G}) = n - 2$ if and only if $G$ is one of the graphs in Figure \ref{fig:book}.
         \item $\operatorname{Z}(\overline{G}) = n - 4$ if and only if $G$ contains at 4-cycle and at least two adjacent vertices on some $C_4$ each have a neighbor not on $C_4$.
         \item $\operatorname{Z}(\overline{G}) = n - 3$, otherwise.
     \end{enumerate}
\end{thm}

\section{A Possible Future Direction}\label{sec:Future}
%CONNECT WITH GRAPH COMPLEMENT CONJECTURE. 

Of much interest to the community is the so-called  graph complement conjecture for minimum rank and, equivalently, maximum nullity \cite{AIM}. For a given graph $G$ on $n$ vertices, we let $S(G)$ denote the set of all $n\times n$ real symmetric matrices $A=[a_{ij}]$, where for $i\neq j$, $a_{ij}\neq 0$ if and only if vertex $i$ is adjacent to vertex $j$. For a graph $G$ on $n$ vertices, we let $mr(G)$ denote the smallest possible rank over all matrices in $S(G)$, and we let $M(G) = n-mr(G)$ be the largest possible nullity over all matrices in $S(G)$. The graph complement conjecture is presented below.

\begin{conj}
For any graph $G$, $mr(G) + mr(\overline{G}) \le |G| + 2$ and equivalently, $M(G) + M(\overline{G})\ge |G|-2$.
\end{conj}

Recent work on partial $k$-trees and their complements (\cite{CharacterizationKtrees,MinRankOfPartialKtrees,ComplementPartialKtrees})  has made progress on the graph complement conjecture much more feasible since every graph $G$ can be viewed as a partial $k$-tree for some $k$. While a partial $k$-tree is constructed by starting with a $k$-tree and deleting edges, the complement involves adding edges so the underlying structure is never lost. The authors in \cite{MinRankOfPartialKtrees} proved the following important result.

\begin{thm}\label{thm:MinrankPartialKtree}
If $G$ is a partial $k$-tree, then $M_+(\overline{G}) \ge |G| - k - 2$. In particular, $mr_+(\overline{G}) = 
k + 2$, where $mr_{+}(G)$ is defined to be the smallest possible rank over all positive semidefinite matrices in $S(G)$ and $M_{+}(G)=n-mr_{+}(G)$.
\end{thm} 

Note that this result is stated for positive semi-definite zero forcing and an analogous result is similarly attained for standard zero forcing.  Therefore, with this work towards proving the graph complement conjecture, it is worthwhile to continue the study of the zero forcing number of the complement of a graph by studying it for partial $k$-trees. 

A first step in this study of partial $k$-trees would be to study $k$-trees. In \cite{CharacterizationKtrees}, a characterization of $k$-trees was given that would be a useful starting point.

\begin{thm}
Let $G$ be a graph with at least $k+1$ vertices. Then, $G$ is a $k$-tree if and only if the following are true:
\begin{enumerate}
    \item $G$ has no $K_{k+2}-$minor
    \item $G$ does not contain any chordless cycle of length at least $4$.
    \item $G$ is $k$-connected. 
\end{enumerate}
\end{thm}

Starting with $2$-tree's, one would explore graphs $G$ such that $G$ has no $K_4$-minor, $G$ does not contain any chordless cycle of length at least 4, and $G$ is $2$-connected. Therefore, towards bounding the zero forcing number of $2$-tree's, a tangible direction that is similar to our result in Theorem \ref{thm:SubKrs} where we are able to bound the zero forcing below for graphs with no $K_{n,m}$ subgraph is be to bound the zero forcing number of the complement of graphs with no $K_4$ minors.

% \subsection{Bicyclic Graphs}

% \begin{defn}
% A graph $G$ of order $n$ is called a {\em bicyclic graph} if $G$ is connected and the number of edges of $G$ is $n + 1$.
% \end{defn}

% \begin{obs}
% If $G$ is a bicyclic graph, then $n - 5 \leq Z(\overline{G}) \leq n -1$. 
% \end{obs}
% \begin{proof}
% By their definition, $G$ is a unicyclic graph with an additional edge. Thus, because $-1 \leq Z(G) - Z(G - e) \leq 1$ and due to the results of Section \ref{sec:uni}, the result follows. 
% \end{proof}

\section*{Acknowledgements}
We thank Michael Young for very helpful discussions and insights which lead to the formulation of Theorem \ref{thm:SubKrs} and some of the resulting consequences.

This work started at the MRC workshop “Finding Needles in Haystacks: Approaches to Inverse
Problems using Combinatorics and Linear Algebra”, which took place in June 2021 with support from the
National Science Foundation and the American Mathematical Society. The authors are grateful to the
organizers of this meeting. In particular, this material is based upon work supported by the National Science Foundation under Grant Number DMS 1916439%1641020.

Shaun M. Fallat was also supported in part by an NSERC Discovery Research Grant, Application No.: RGPIN--2019--03934.
%%%%%%%%%%%%%%%%%%%%%%%%%%%%%%%%%%%%%%%%%%%%%%%%%%%%%%%%%%%%%%%%%%%%%%%%%

%%%%%%%%%%%%%%%%%%%%%%%%%%%%%%%%%%%%%%%%%%%%%%%%%%%%%%%%%%%%%%%%%%%%%%%%%

\end{document}